\documentclass[11pt]{article}

\usepackage{amssymb}
\usepackage{graphicx}
\usepackage{color}
\usepackage{amsmath,amsthm,amsfonts,epsfig,setspace}
\usepackage{float} 
\newcommand{\ds}{\displaystyle}
\def\nm{\noalign{\medskip}}
\newtheorem{thm}{Theorem}

\newtheorem{lem}{Lemma}

 \def\p{\partial}
\def \Vh0{\stackrel{\circ}{V}_h} 
   
  \def\om{\omega}
   
\def\l{\label}  \def\f{\frac} \def\df{\dfrac} 

   \def\eps{\varepsilon}

\def\l|{\left|}
\def\r|{\right|}

\newcommand{\R}{\mathbb{R}}
\newcommand{\N}{\mathbb{N}}
\newcommand{\Z}{\mathbb{Z}}

\newcommand{\lc}
{\mathrel{\raise2pt\hbox{${\mathop<\limits_{\raise1pt\hbox
{\mbox{$\sim$}}}}$}}}

\newcommand{\gc}
{\mathrel{\raise2pt\hbox{${\mathop>\limits_{\raise1pt\hbox{\mbox{$\sim$}}}}$}}}

\newcommand{\ec}
{\mathrel{\raise2pt\hbox{${\mathop=\limits_{\raise1pt\hbox{\mbox{$\sim$}}}}$}}}

\def\be{\begin{equation}} \def\ee{\end{equation}}

\def\bea{\begin{eqnarray}}  \def\eea{\end{eqnarray}}

\def\beas{\begin{eqnarray*}} \def\eeas{\end{eqnarray*}}

\def\bn{\begin{enumerate}} \def\en{\end{enumerate}}

\def\bd{\begin{description}} \def\ed{\end{description}}

\title{Theory of plasmonic metasurfaces\thanks{\footnotesize This work was supported  by the ERC Advanced Grant Project MULTIMOD--267184. Hai Zhang was supported by a startup fund from HKUST.}}

\author{
Habib Ammari\thanks{\footnotesize Department of Mathematics, 
ETH Z\"urich, R\"amistrasse 101, CH-8092 Z\"urich, Switzerland (habib.ammari@math.ethz.ch,  wei.wu@sam.math.ethz.ch, sanghyeon.yu@math.ethz.ch)} \and Matias Ruiz\thanks{\footnotesize Department of Mathematics and Applications,
Ecole Normale Sup\'erieure, 45 Rue d'Ulm, 75005 Paris, France
(matias.ruiz@ens.fr).} \and 
 Wei Wu\footnotemark[2] \and Sanghyeon Yu\footnotemark[2],
\and  
Hai Zhang\thanks{\footnotesize
Department of Mathematics, 
 HKUST,  Clear Water Bay, Kowloon, Hong Kong (haizhang@ust.hk).}
}

\date{}

\begin{document}

\maketitle


\begin{abstract}
In this paper we derive an impedance boundary condition to approximate the optical scattering effect of an array of plasmonic nanoparticles mounted on a perfectly conducting plate. We show that at some resonant frequencies the impedance blows up, allowing for a significant reduction of the scattering from the plate. Using the spectral properties of a Neumann-Poincar\'e type operator, we investigate the dependency of the impedance with respect to changes in the nanoparticle geometry and configuration. 
\end{abstract}


\bigskip

\noindent {\footnotesize Mathematics Subject Classification
(MSC2000): 35R30, 35C20.}



\section{Introduction} \label{sec-intro}

Driven by the search for new materials with interesting and unique optical properties, the field of plasmonic nanoparticles  has grown immensely in the last decade \cite{link}. Recently, there have been several interesting mathematical works on plasmonic resonances for nanoparticles
\cite{pierre, matias, matias2, kang1, hyeonbae, triki, Gri12, kang3}. 
On the other hand, scattering of waves by periodic structures plays a central role in optics \cite{garcia}.   

In this paper we consider the scattering by a layer of periodic plasmonic nanoparticles mounted on a perfectly conducting sheet. We design the layer in order to control and transform waves. Since the thickness of the layer, which is of the same order of the diameter of the individual nanoparticles, is negligible compared to the  
wavelength, it can be approximated by an impedance boundary condition. 
Our main result is to prove that at some resonant frequencies, which are fully characterized in terms of the periodicity, the shape and the material parameters of the nanoparticles, the thin layer has anomalous reflection properties and can be viewed as a metasurface. Since the period of the array is much smaller than the wavelength, the resonant frequencies of the array of nanoparticles differ significantly from those of single nanoparticles. As shown in this paper, they are associated with eigenvalues of a periodic Neumann-Poincar\'e type operator. In contrast with quasi-static plasmonic resonances of single nanoparticles, they depend  on the particle size. For simplicity, only one-dimensional arrays embedded in $\mathbb{R}^2$ are considered in this paper. The extension to the two-dimensional case is straightforward and the dependence of the plasmonic resonances on the parameters of the lattice is easy to derive.   

The array of plasmonic nanoparticles can be used to efficiently reduce the scattering of the perfectly conducting sheet. We present numerical results to illustrate  our main findings in this paper, which open a door for a mathematical
and numerical framework for realizing full control of waves using metasurfaces \cite{alu,pnas,proca}. Our approach applies to any example of periodic distributions of resonators having resonances in the quasi-static regime. It provides a framework for explaining the observed extraordinary or meta properties of such structures and for optimizing these properties.   

The paper is organized as follows. We first use  formulate the problem of approximating the effect of a thin layer with impedance boundary conditions and give useful results on the 1-d periodic Green function. Then we derive the effective impedance boundary conditions and give the shape derivative of the impedance parameter. In doing so,  
we analyze the spectral properties of the 1-d periodic Neumann-Poincar\'e operator defined by (\ref{periodicNP}) and obtain an explicit formula for the equivalent boundary condition in terms of its eigenvalues and eigenvectors. Finally, we illustrate with a few numerical experiments the anomalous change in the equivalent impedance boundary condition due to the plasmonic resonances of the periodic array of nanoparticles. For simplicity, we only consider the scalar wave equation and use a two-dimensional setup. The results of this paper can be readily generalized to higher dimensions as well as to the full Maxwell equations.

\section{Setting of the problem} \label{sect1}
We use the Helmholtz equation to model the propagation of light. This approximation can be viewed as a special case of Maxwell's equations, when the incident wave $u^i$ is transverse magnetic (TM) or transverse electric (TE) polarized.

Consider a particle occupying a bounded domain $D\Subset\mathbb{R}^2$ of class $\mathcal{C}^{1,\alpha}$ for some $0<\alpha<1$ and with size of order $\delta \ll 1$. The particle is characterized by electric permittivity $\varepsilon_c$ and magnetic permeability $\mu_c$, both of which may depend on the frequency of the incident wave. Assume that $\Im m\, \eps_c >0, \Re e\, \mu_c <0, \Im m \, \mu_c >0$ and define
\beas
k_m = \omega \sqrt{\eps_m \mu_m}, \quad k_c = \omega \sqrt{\eps_c \mu_c}, 
\eeas
where $\eps_m$ and $\mu_m$ are the permittivity and permeability of free space respectively and $\om$ is the frequency. Throughout this paper, we assume that $\eps_m$ and $\mu_m$ are real and positive and $k_m$ is of order $1$.

We consider the configuration shown in Figure \ref{figPeriodicNanoparticles}, where a particle $D$ is repeated periodically in the $x_1$-axis with period $\delta$, and is of a distance of order $\delta$ from the boundary $x_2=0$ of the half-space
$\R^2_+ := \{(x_1,x_2)\in \R^2,\; x_2 >0 \}$. We denote by $\mathcal{D}$ this collection of periodically arranged particles and $\Omega :=  \R^2_+ \setminus \overline{\mathcal{D}}$.
\begin{figure}[!h]
\begin{center}
\includegraphics[scale=0.25]{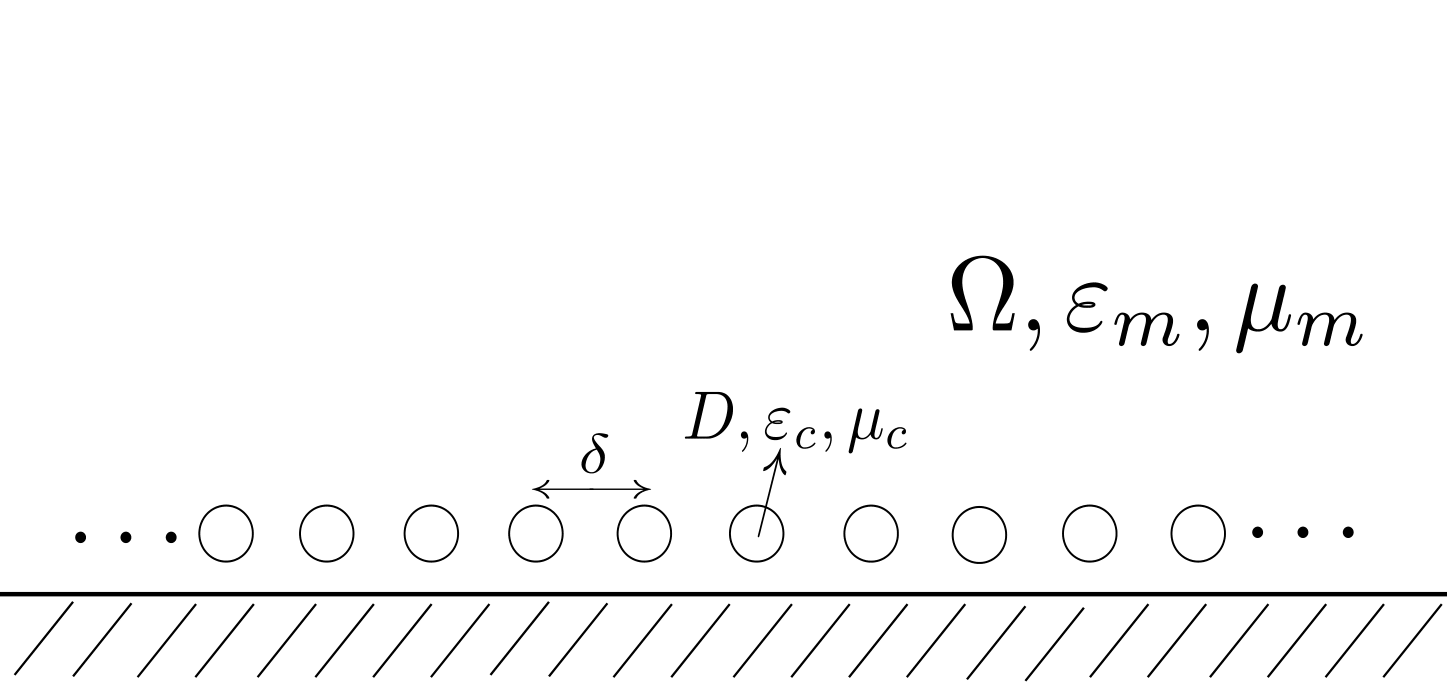}
\caption{ \label{figPeriodicNanoparticles} Thin layer of nanoparticles in the half space.}
\end{center}
\end{figure}

Let $u^i(x)=e^{ik_m d\cdot x}$ be the incident wave. Here, $d$ is the unit incidence direction. The scattering problem is modeled as follows
\begin{equation} \label{eq-problem setting meta}
 \left\{
\begin{array} {ll}
&\ds \nabla \cdot \f{1}{\mu_{\mathcal{D}}} \nabla  u+ \omega^2 \eps_{\mathcal{D}} u  = 0 
\quad \mbox{in } \R^2_+ \setminus \partial \mathcal{D}, \\
\nm
& u_{+} -u_{-}  =0   \quad \mbox{on } \partial \mathcal{D}, \\
\nm
&  \ds \f{1}{\mu_{m}} \f{\p u}{\p \nu} \bigg|_{+} - \f{1}{\mu_{c}} \f{\p u}{\p \nu} \bigg|_{-} =0 \quad \mbox{on } \partial \mathcal{D},\\
\nm
&  u- u^{i}  \,\,\,  \mbox{satisfies an outgoing radiation condition at infinity},\\
\nm
& u = 0 \quad \mbox{on } \partial \R^2_+ = \{(x_1,0),\; x_1\in \R\},
  \end{array}
 \right.
\end{equation}
where
\beas
\eps_{\mathcal{D}}= \eps_m \chi(\Omega) + \eps_c \chi({\mathcal{D}}), \quad
\mu_{\mathcal{D}} = \eps_m \chi(\Omega)+ \eps_c \chi({\mathcal{D}}),
\eeas
and ${\p }/{\p \nu}$ denotes the outward normal derivative on $\partial \mathcal{D}$.

Following \cite{abboud}, under the assumption that the wavelength of the incident wave is much larger than the size of the nanoparticle, a certain homogenization occurs, and we can construct $z \in \mathbb{C}$ such that the solution to
\be
\label{approxz}
 \left\{
\begin{array} {ll}
&\ds \Delta  u_{\mathrm{app}} + k_m^2 u_{\mathrm{app}} = 0 \quad \textnormal{in } \R^2_+, \\
\nm
& u_{\mathrm{app}} + \delta z\f{\p u_{\mathrm{app}}}{\p x_2} = 0 \quad \textnormal{on } \partial \R^2_+, \\
\nm
&  u_{\mathrm{app}}-{u^i}  \,\,\,  \textnormal{satisfies outgoing radiation condition at infinity},
  \end{array}
 \right.
\ee
gives the leading order approximation for  $u$.  We will refer to $u_{\mathrm{app}} + \delta z \f{\p u_{\mathrm{app}}}{\p x_2} = 0$ as the equivalent impedance boundary condition for problem \eqref{eq-problem setting meta}. 


\section{1-d periodic Green function} \label{Gper}

Consider the function $G_{\sharp}: \R^2 \rightarrow \mathbb{C}$ satisfying
\be \label{eq-Lap Green periodic meta}
\Delta G_{\sharp}(x) = \sum_{n \in \Z} \delta(x + (n,0)).
\ee
We call $G_{\sharp}$ the 1-d periodic Green function for $\R^2$. 
\begin{lem} Let $x=(x_1,x_2)$, then
\beas
G_{\sharp}(x) = \f{1}{4\pi}\log\big( \sinh^2(\pi x_2) + \sin^2(\pi x_1) \big),
\eeas
satisfies \eqref{eq-Lap Green periodic meta}.
\end{lem}
\begin{proof} We have
\bea \label{eq-poisson Green periodic meta}
\Delta G_{\sharp}(x) &=& \sum_{n \in \Z} \delta(x + (n,0))\nonumber \\
&=& \sum_{n \in \Z}\delta(x_2) \delta(x_1 + n)\nonumber \\
&=& \sum_{n \in \Z}\delta(x_2) e^{i 2\pi n x_1},
\eea
where we have used the Poisson summation formula $\sum_{n \in \Z}\delta(x_1 + n) = \sum_{n \in \Z} e^{i 2\pi n x_1}$.

On the other hand, since $G_{\sharp}$ is periodic in $x_1$ of period $1$, we have
\beas
G_{\sharp}(x) = \sum_{n \in \Z}\beta_n(x_2)e^{i 2\pi n x_1},
\eeas
therefore
\be \label{eq-fourier Green periodic meta}
\Delta G_{\sharp}(x) = \sum_{n \in \Z}(\beta_n^{''}(x_2) + (i 2\pi n)^2\beta_n)e^{i 2\pi n x_1}.
\ee
Comparing \eqref{eq-poisson Green periodic meta} and \eqref{eq-fourier Green periodic meta} yields
\beas
\beta_n^{''}(x_2) + (i 2\pi n)^2\beta_n = \delta(x_2).
\eeas
A solution to the previous equation can be found by using standard techniques for ordinary differential equations. We have
\beas
\beta_0 &=& \f{1}{2}|x_2| + c, \\
\beta_n &=& \f{-1}{4 \pi |n|}e^{-2 \pi |n||x_2|}, \quad n\neq 0,
\eeas
where $c$ is a constant. Subsequently,
\beas
G_{\sharp}(x) &=& \f{1}{2}|x_2| + c - \sum_{n \in \Z\backslash\{0\}}\f{1}{4 \pi |n|}e^{-2 \pi |n||x_2|}e^{i 2\pi n x_1}\\
&=& \f{1}{2}|x_2| + c - \sum_{n \in \N\backslash\{0\}}\f{1}{2 \pi n}e^{-2 \pi n |x_2|}\cos(2\pi n x_1)\\
&=& \f{1}{4\pi}\log\big( \sinh^2(\pi x_2) + \sin^2(\pi x_1) \big),
\eeas
where we have used the summation identity (see, for instance, \cite[pp. 813-814]{Hale/Lunel})
\beas
\sum_{n \in \N\backslash\{0\}}\f{1}{2 \pi n}e^{-2 \pi n |x_2|}\cos(i 2\pi n x_1) = \f{1}{2}|x_2| - \f{\log(2)}{2\pi} \\ 
- \f{1}{4\pi}\log\big( \sinh^2(\pi x_2) + \sin^2(\pi x_1) \big),
\eeas
and defined $c = -\df{\log(2)}{2\pi}$.
\end{proof}

Throughout, we denote by $H^s(\p B)$ the usual Sobolev space of order $s$ on $\p B$ and by $Id$ the identity operator. Let us also denote by $G_{\sharp}(x,y):= G_{\sharp}(x-y)$. In the following we define the 1-d periodic single layer potential and 1-d periodic Neumann-Poincar\'e operator, respectively, for a bounded domain $B\Subset\big(-\df{1}{2},\df{1}{2}\big)\times\mathbb{R}$ which we assume to be of class $\mathcal{C}^{1,\alpha}$ for some $0<\alpha<1$. Let
\bea
\mathcal{S}_{B\sharp}: H^{-\f{1}{2}}(\p B) &\longrightarrow & H^{1}_{\mathrm{loc}}(\R^2), H^{\f{1}{2}}(\p B) \nonumber\\
\varphi &\longmapsto & \mathcal{S}_{B,\sharp} [\varphi](x) = \int_{\p B} G_{\sharp}(x, y) \varphi(y) d\sigma(y) \nonumber
\eea
for  $x \in \mathbb{R}^2, x \in \p B$
and let
\bea 
\mathcal{K}_{B\sharp}^*: H^{-\f{1}{2}}(\p B) &\longrightarrow & H^{-\f{1}{2}}(\p B) \nonumber \\
\varphi &\longmapsto & \mathcal{K}_{B,\sharp}^* [\varphi](x) = \int_{\p B } \f{\p G_{\sharp}(x, y)}{ \p \nu(x)} \varphi(y) d\sigma(y)  \nonumber
\eea
for $x \in \p B$. As in \cite{shapiro}, the periodic Neumann-Poincar\'e operator can be symmetrized. The following lemma holds. 
\begin{lem} \label{lem-properties periodic LP meta}
\begin{enumerate}
\item [(i)]
For any $\varphi\in H^{-\f{1}{2}}(\p B)$, $\mathcal{S}_{B\sharp}[\varphi]$ is harmonic in $B$ and in $\big(-\df{1}{2},\df{1}{2}\big) \times\mathbb{R}\backslash \overline{B}$;
\item[(ii)]
 The following trace formula holds: for any $\varphi\in H^{-\f{1}{2}}(\p B)$, 
$$
(-\f{1}{2}Id+\mathcal{K}_{B\sharp}^*)[\varphi] = \f{\p \mathcal{S}_{B\sharp}[\varphi]}{\p \nu}\Big\vert_-;
$$
\item[(iii)] The following Calder\'on identity holds:
$\mathcal{K}_{B\sharp} \mathcal{S}_{B\sharp}= \mathcal{S}_{B\sharp}\mathcal{K}_{B\sharp}^*$, where $\mathcal{K}_{B\sharp}$ is the $L^2$-adjoint of $\mathcal{K}_{B\sharp}^*$;
\item[(iv)]
The operator $\mathcal{K}_{B\sharp}^*:H^{-\f{1}{2}}_0(\p B)\rightarrow H^{-\f{1}{2}}_0(\p B) $ is compact self-adjoint equipped with the following
inner product
\be \label{innerproduct}
(u, v)_{\mathcal{H}^*_0}= - (u, \mathcal{S}_{B\sharp}[v])_{-\f{1}{2},\f{1}{2}}
\ee
with $(\cdot, \cdot)_{-\f{1}{2}, \f{1}{2}}$ being the duality pairing between $H^{-\f{1}{2}}_0(\p B)$ and  $H^{\f{1}{2}}_0(\p B)$, which makes $\mathcal{H}^*_0 $ equivalent to $H^{-\f{1}{2}}_0(\p B)$. Here, by $E_0$ we denote the zero-mean subspace of $E$.
\item[(v)] Let $(\lambda_j,\varphi_j) $, $j = 1, 2, \ldots$ be the eigenvalue and normalized eigenfunction pair of $\mathcal{K}_{B\sharp}^*$ in $\mathcal{H}^*_0(\p B)$,
then $\lambda_j \in (-\f{1}{2}, \f{1}{2})$ and $\lambda_j \rightarrow 0$ as $j \rightarrow \infty$.
\end{enumerate}
\end{lem}
\begin{proof}
First, note that a Taylor expansion of $\sinh^2(\pi x_2) + \sin^2(\pi x_1)$ yields
\beas
G_{\sharp}(x) = \f{\log|x|}{2\pi} + R(x),
\eeas
where $R$ is a smooth function such that
\beas
R(x) = \f{1}{4\pi}\log(1+O(x_2^2-x_1^2)).
\eeas
We can decompose the operators $\mathcal{S}_{B\sharp}$ and $\mathcal{K}_{B\sharp}^*$ on $\mathcal{H}^*_0(\p B)$ accordingly. We have
\beas
\mathcal{S}_{B\sharp} = \mathcal{S}_{B} + \mathcal{G}_B, \quad
\mathcal{K}_{B\sharp}^* = \mathcal{K}_{B}^* + \mathcal{F}_B ,
\eeas
where $\mathcal{S}_{B}$ and $\mathcal{K}_{B}^*$ are the single layer potential and Neumann-Poincar\'e operator (see \cite{book2}), respectively, and $\mathcal{G}_B,\mathcal{F}_B$ are smoothing operators.
Using this fact, the proof of the Lemma follows the same arguments as those given in \cite{book3, book2}.
\end{proof}

\section{Boundary layer corrector and effective impedance} \label{sect2}

In order to compute $z$, we introduce the following asymptotic expansion \cite{abboud, allaire}:
\be \label{eq-expansion u_s meta}
u = u^{(0)} + u_{BL}^{(0)} + \delta (u^{(1)} + u_{BL}^{(1)}) + ...
\ee
where the leading-order term $u^{(0)}$ is solution to
\beas
 \left\{
\begin{array} {ll}
&\ds \Delta  u^{(0)} + k_m^2 u^{(0)}  = 0 \quad \mbox{in } \R^2_+, \\
\nm
& u^{(0)} = 0 \quad \mbox{on } \partial \R^2_+, \\
\nm
&  u^{(0)}-{u^i}  \,\,\,  \textnormal{satisfies an outgoing radiation condition at infinity}.
  \end{array}
 \right.
\eeas
The boundary-layer correctors  $u_{BL}^{(0)}$ and $u_{BL}^{(1)}$ have to be exponentially decaying in the $x_2$-direction. Note that 
according to \cite{abboud, allaire},
$u_{BL}^{(0)}$ is introduced in order to correct (up to the first-order in $\delta$) the transmission condition on the boundary of the nanoparticles, which is not satisfied by the leading-order term $u^{(0)}$ in the asymptotic expansion of $u$, while  $u_{BL}^{(1)}$ is a higher-order correction term and does not contribute to the first-order equivalent boundary condition in (\ref{approxz}). 

We next construct the corrector  $u_{BL}^{(0)}$. We first introduce a 
function $\alpha$ and a complex constant $\alpha_\infty$ such that they satisfy the rescaled problem:
\be \label{eq-alpha meta}
 \left\{
\begin{array} {ll}
&\ds \Delta \alpha = 0 \quad \mbox{in } \Big(\R^2_+\backslash \overline{\mathcal{B}}\Big) \cup \mathcal{B}  , \\
\nm
& \alpha |_{+} - \alpha |_{-}  =0   \quad \mbox{on } \partial \mathcal{B}, \\
\nm
&  \ds \f{1}{\mu_{m}} \f{\p \alpha}{\p \nu} \bigg|_{+} - \f{1}{\mu_{c}} \f{\p \alpha}{\p \nu} \bigg|_{-} = \Big(\f{1}{\mu_{c}}-\f{1}{\mu_{m}}\Big)\nu_2 \quad \mbox{on } \partial \mathcal{B},\\
\nm
&  \alpha = 0 \quad \mbox{on } \partial \R^2_+, \\
\nm & \alpha - \alpha_\infty \mbox{ is exponentially decaying as } x_2 \rightarrow +\infty.
  \end{array}
 \right.
\ee
Here, $\nu = (\nu_1,\nu_2)$ and
$B =  D/ {\delta}$ is repeated periodically in the $x_1$-axis with period $1$ and $\mathcal{B}$ is the collection of these periodically arranged particles.

Then $u_{BL}^{(0)}$ is defined by 
$$
u_{BL}^{(0)}(x):=\delta \f{\p u^{(0)}}{\p x_2}(x_1,0)\left(\alpha(\f{x}{\delta})-\alpha_{\infty}\right).
$$
The corrector $u^{(1)}$ can be found to be the solution to
\beas
 \left\{
\begin{array} {ll}
&\ds \Delta  u^{(1)} + k_m^2 u^{(1)}  = 0 \quad \mbox{in } \R^2_+, \\
\nm
& u^{(1)} = \alpha_\infty \frac{\partial u^{(0)}}{\partial x_2} \quad \mbox{on } \partial \R^2_+, \\
\nm
&  u^{(1)}  \,\,\,  \textnormal{satisfies an outgoing radiation condition at infinity}.
  \end{array}
 \right.
\eeas

%

By writing 
\be \label{eq-expansion u_app meta}
u_{\mathrm{app}} = u^{(0)} + u_{BL}^{(0)}+ \delta u^{(1)},
\ee
we arrive at (\ref{approxz}) with $z= - \alpha_\infty$, up to a second order term in $\delta$. We summarize the above results in the following theorem. 
\begin{thm}
The solution $u_{\mathrm{app}}$ to (\ref{approxz}) with $z= - \alpha_\infty$ approximates pointwisely (for $x_2>0$) the exact solution $u$ to (\ref{eq-problem setting meta}) as $\delta \rightarrow 0$, up to a second order term in $\delta$. 
\end{thm}

In order to compute $\alpha_\infty$, we derive an integral representation for the solution  $\alpha$ to \eqref{eq-alpha meta}.  
We make use of the periodic Green function $G_{\sharp}$ defined by (\ref{eq-Lap Green periodic meta}). Let 
\beas
G_{\sharp}^+(x, y) = G_{\sharp}\big((x_1-y_1, x_2-y_2)\big)-G_{\sharp}\big((x_1-y_1, -x_2-y_2)\big),
\eeas
which is the periodic Green's function in the upper half space with Dirichlet boundary conditions, 
and define
\beas
\mathcal{S}_{B\sharp}^+: H^{-\f{1}{2}}(\p B) &\longrightarrow & H^{1}_{\mathrm{loc}}(\R^2), H^{\f{1}{2}}(\p B) \\
\varphi &\longmapsto & \mathcal{S}_{B,\sharp}^+ [\varphi](x) = \int_{\p B} G_{\sharp}^+(x, y) \varphi(y) d\sigma(y)
\eeas
for $x \in \mathbb{R}^2_+, x \in \p B$ and 
\begin{equation} \label{periodicNP}
\begin{array}{l}
\ds (\mathcal{K}_{B\sharp}^*)^+: H^{-\f{1}{2}}(\p B) \longrightarrow  H^{-\f{1}{2}}(\p B) \\
\ds \varphi  \longmapsto  (\mathcal{K}_{B,\sharp}^*)^+ [\varphi](x) = \int_{\p B } \f{\p G_{\sharp}^+(x, y)}{ \p \nu(x)} \varphi(y) d\sigma(y)
\end{array}
\end{equation}
for $ x \in \p B$. 

It is clear that the results of Lemma \ref{lem-properties periodic LP meta} hold true for $\mathcal{S}_{B\sharp}^+$ and $(\mathcal{K}_{B\sharp}^*)^+$. Moreover, for any $\varphi\in H^{-\f{1}{2}}(\p B)$, we have
\beas
\mathcal{S}_{B,\sharp}^+ [\varphi](x) = 0 \quad \mbox{for } x \in \partial \R^2_+.
\eeas

Now, we can readily see that $\alpha$ can be represented as $\alpha = \mathcal{S}_{B,\sharp}^+[\varphi]$, where 
$\varphi\in H^{-\f{1}{2}}(\p B)$ satisfies
\beas
\f{1}{\mu_{m}} \f{\p \mathcal{S}_{B,\sharp}^+[\varphi]}{\p \nu} \bigg|_{+} - \f{1}{\mu_{c}} \f{\p \mathcal{S}_{B,\sharp}^+[\varphi]}{\p \nu} \bigg|_{-} = \Big(\f{1}{\mu_{c}}-\f{1}{\mu_{m}}\Big)\nu_2 \quad \mbox{on }\p B.
\eeas
Using the jump formula from Lemma \ref{lem-properties periodic LP meta}, we arrive at
\beas
\big(\lambda_{\mu}Id - (\mathcal{K}_{B\sharp}^*)^+\big)[\varphi] = \nu_2,
\eeas
where
\beas
\lambda_{\mu} = \f{\mu_c+\mu_m}{2(\mu_c-\mu_m)}.
\eeas
Therefore, we obtain that
\beas
\alpha = \mathcal{S}_{B,\sharp}^+\big(\lambda_{\mu}Id - (\mathcal{K}_{B\sharp}^*)^+\big)^{-1}[\nu_2].
\eeas

\begin{lem} \label{lem-alpha assymp meta} 
Let $x=(x_1,x_2)$. Then, for $x_2 \rightarrow +\infty$, the following asymptotic expansion holds:
\beas
\alpha = \alpha_{\infty} + O(e^{-x_2}),
\eeas
with
\beas
\alpha_{\infty} = -\int_{\p B} y_2 \big(\lambda_{\mu}Id - (\mathcal{K}_{B\sharp}^*)^+\big)^{-1}[\nu_2](y) d\sigma(y).
\eeas
\end{lem}
\begin{proof}
The result follows from an asymptotic analysis of $G_{\sharp}^+(x, y)$. Indeed, suppose that $x_2 \rightarrow +\infty$, we have
$$\begin{array}{l}
G_{\sharp}^+(x, y) = \f{1}{4\pi}\log\big( \sinh^2(\pi (x_2-y_2) ) + \sin^2(\pi (x_1-y_1) ) \big) \\
\nm \ds  - \f{1}{4\pi}\log\big( \sinh^2(\pi (x_2+y_2) ) + \sin^2(\pi (x_1-y_1)) \big)\\
\nm \ds  = \f{1}{4\pi}\log\big( \sinh^2(\pi (x_2-y_2) )\big) \\
\nm \ds   - \f{1}{4\pi}\log\big( \sinh^2(\pi (x_2+y_2) )\big)  \\
\nm \ds  + O\big(\log\left(1 + \f{1}{\sinh^2(x_2)}\right) \big)\\
\nm \ds  = \f{1}{2\pi}\bigg(\log\Big(\f{e^{\pi (x_2-y_2)}-e^{-\pi (x_2+y_2) }}{2}\Big) \\
\nm \ds  - \log\Big(\f{e^{\pi (x_2+y_2) }-e^{-\pi (x_2-y_2) }}{2}\Big)\bigg) + O\big(\log\left(1 + e^{-x_2^2}\right) \big)\\
\nm \ds  = -y_2 + O(e^{-x_2}),
\end{array}
$$
which yields the desired result. 
\end{proof}

Finally, it is important to note that $\alpha_{\infty}$ depends on the geometry 
and size of the particle $B$.  

Since $(\mathcal{K}_{B\sharp}^*)^+ : \mathcal{H}^*_0 \rightarrow \mathcal{H}^*_0$ is a compact self-adjoint operator, where $\mathcal{H}^*_0$ is defined as in Lemma \ref{lem-properties periodic LP meta}, we can write
\beas
\alpha_{\infty} &=& -\int_{\p B} y_2 \big(\lambda_{\mu}Id - (\mathcal{K}_{B\sharp}^*)^+\big)^{-1}[\nu_2](y) d\sigma(y),\\
&=& -\int_{\p B} y_2 \sum_{j=1}^{\infty}\f{(\varphi_j,\nu_2)_{\mathcal{H}^*_0}\varphi_j(y)}{\lambda_{\mu}-\lambda_{j}} d\sigma(y),\\
&=& \sum_{j=1}^{\infty}\f{(\varphi_j,\nu_2)_{\mathcal{H}^*_0}(\varphi_j,y_2)_{-\f{1}{2},\f{1}{2}}}{\lambda_{\mu}-\lambda_{j}},\\
\eeas
where $\lambda_1,\lambda_2,\dots$ are the eigenvalues of $(\mathcal{K}_{B\sharp}^*)^+$ and $\varphi_1,\varphi_2,\dots$ is a corresponding orthornormal basis of eigenvectors.

On the other hand, by integrating by parts we get $$(\varphi_j,y_2)_{-\f{1}{2},\f{1}{2}} = 
\frac{1}{\frac{1}{2} - \lambda_j} (\varphi_j,\nu_2)_{
\mathcal{H}^*_0}.$$ This together with the fact that
$\Im m \,  \lambda_\mu <0$ (by the Drude model \cite{pierre}), 
yield the following lemma.  
\begin{lem} We have $\Im m\, \alpha_\infty >0$.
\end{lem}

Finally, we give a formula for the shape derivative \cite{akl} of $\alpha_\infty$. This formula can be used to optimize $|\alpha_\infty|$ , for a given frequency $\omega$, in terms of the shape $B$ of the nanoparticle.   Let $B_\eta$ be an $\eta$-perturbation of $B$; i.e., let $h \in \mathcal{C}^1(\partial B)$ and $\partial B_\eta$ be given by 
$$
\partial B_\eta =\bigg\{ x + \eta h(x) \nu(x), x \in \partial B\bigg\}. 
$$
Following \cite{zribi} (see also \cite{book3}), we can prove that
$$\begin{array} {lll}
\alpha_\infty(B_\eta) &=&\ds \alpha_\infty(B) + \eta (\frac{\mu_m}{\mu_c} -1) \\
\nm
&& \ds \times \int_{\partial B} h \bigg[ 
\frac{\partial v}{\partial \nu} \big|_-    \frac{\partial w}{\partial \nu} \big|_-   + \frac{\mu_c}{\mu_m} \frac{\partial v}{\partial \tau} \big|_-    \frac{\partial w}{\partial \tau} \big|_-  \bigg] \, d\sigma,
\end{array}
$$
where  $\partial / \partial \tau$ is the tangential derivative on $\partial B$, $v$ and $w$ periodic with respect to $x_1$ of period $1$ and satisfy  
$$
\left\{
\begin{array} {ll}
&\ds \Delta v = 0 \quad \mbox{in } \Big(\R^2_+\backslash \overline{\mathcal{B}}\Big) \cup \mathcal{B}  , \\
\nm
& v|_{+} - v|_{-}  =0   \quad \mbox{on } \partial \mathcal{B}, \\
\nm
&  \ds  \f{\p v}{\p \nu} \bigg|_{+} - \frac{\mu_m}{\mu_c} \f{\p v}{\p \nu} \bigg|_{-} = 0 \quad \mbox{on } \partial \mathcal{B},\\
\nm
&  v- x_2 \rightarrow 0  \quad \mbox{as }  x_2 \rightarrow +\infty,
  \end{array}
 \right.
$$
and 
$$
\left\{
\begin{array} {ll}
&\ds \Delta w = 0 \quad \mbox{in } \Big(\R^2_+\backslash \overline{\mathcal{B}}\Big) \cup \mathcal{B}  , \\
\nm
& \frac{\mu_m}{\mu_c} w|_{+} -  w|_{-}  =0   \quad \mbox{on } \partial \mathcal{B}, \\
\nm
&  \ds \f{\p w}{\p \nu} \bigg|_{+} -  \f{\p w}{\p \nu} \bigg|_{-} = 0 \quad \mbox{on } \partial \mathcal{B},\\
\nm
&  w- x_2 \rightarrow 0  \quad \mbox{as }  x_2 \rightarrow +\infty,
  \end{array}
 \right.
$$
respectively. Therefore, the following lemma holds. 
\begin{lem}
The shape derivative $d_S \alpha_\infty(B)$ of $\alpha_\infty$ is given by 
$$
d_S \alpha_\infty(B) =\ds (\frac{\mu_m}{\mu_c} -1)  \bigg[\frac{\partial v}{\partial \nu} \big|_-    \frac{\partial w}{\partial \nu} \big|_-  
 + \frac{\mu_c}{\mu_m} \frac{\partial v}{\partial \tau} \big|_-    \frac{\partial w}{\partial \tau} \big|_- \bigg].
$$
\end{lem}
If we aim to maximize the functional $J:= \frac{1}{2} |\alpha_\infty|^2$ over $B$, then it can be easily seen that $J$ is Fr\'echet differentiable and its Fr\'echet derivative is given by
$
\Re e \, d_S \alpha_\infty(B) \overline{\alpha_\infty(B)}.  
$ 
As in \cite{numer}, in order to include cases where topology changes and multiple components are allowed, a level-set version of the optimization procedure described below can be developed.

\section{Numerical illustrations} \label{sect3}

\subsection{Setup and methods}

Here, we assume that the particles are made of gold and use the Drude model to compute their electric properties as a function of the wavelength. We recall that, from the Drude model \cite{pierre}, the electric properties of the particles depend on the frequency of the incoming wave, or equivalently, the wavelength. The effective impedance $\alpha_{\infty}$ is computed using periodic layer potentials. 

Figure \ref{figCirclesShift meta} shows $|\alpha_{\infty}|$ as a function of the wavelength for disks of different sizes, all centered at $(0,0.5)$.

Figure \ref{figCirclesVerticalShift meta} shows $|\alpha_{\infty}|$ as a function of the wavelength for two 
disks of the same fixed radius equal to $0.2$ but centered at two different distances  from $x_2=0$.

In Figures \ref{figCircleSingleFreq meta} and \ref{figThreeCircleSingleFreq meta} we plot $|\alpha_{\infty}|$ as a function of the wavelength for a disk and a group of three well-separated disks. We can see that a disk can be excited roughly at one single frequency whereas three disks can be excited at different frequencies but with lower values of $|\alpha_{\infty}|$.

\subsection{Results and discussion}
An important conclusion is that the spectrum of the periodic Neumann-Poincar\'e operator defined by (\ref{periodicNP}) varies with the position and size of the particles. Therefore, the resonances of the effective impedance $\alpha_{\infty}$ depend not only on the geometry of the particle $B$ but also on its size and position.
 One can see (Figs. \ref{figCirclesShift meta} and \ref{figCirclesVerticalShift meta}) a change in the magnitude and a shift of the resonances.
The plasmonics resonances shift to smaller wavelengths and the magnitude of the peak value 
 increases with increasing volume. We remark that this is not particular to the examples considered here. In fact, this is the case for any particle. These two phenomena are due to the strong interaction between the particles and the ground that appears as their sizes increase while the period of the arrangement is fixed.

Note also that in our analysis we did not assume the particles to be simply connected. In fact, the theory is still valid for particles which have two or more components. This allows for more possibilities when choosing a particular geometry for the optimization of the effective impedance. For instance, 
one may want to design a geometry such that a single frequency is excited with a very pronounced peak or, on the other hand, to excite not only a specific frequency but rather a group of them.

%
%
%
%

\section{Concluding remarks}

In this paper we have considered the scattering by an array of plasmonic nanoparticles mounted on a perfectly conducting plate and showed both analytically and numerically the significant change in the boundary condition induced by the nanoparticles at their periodic plasmonic frequencies. We have also proposed an optimization approach to maximize this change in terms of the shape of the nanoparticles. Implementation and testing of this approach will be reported elsewhere. Our results in this paper can be generalized in many directions. Different boundary conditions on the plate as well as curved plates can be considered. Our approach can be easily extended to two-dimensional arrays embedded in $\mathbb{R}^3$ and the lattice effect can be included. Full Maxwell's equations to model the light propagation can be used.  The observed extraordinary or meta properties of periodic  distributions of subwavelength resonators can be explained by the  approach proposed in this paper.


\begin{figure}[h!]
\begin{center}
\includegraphics[scale=0.25]{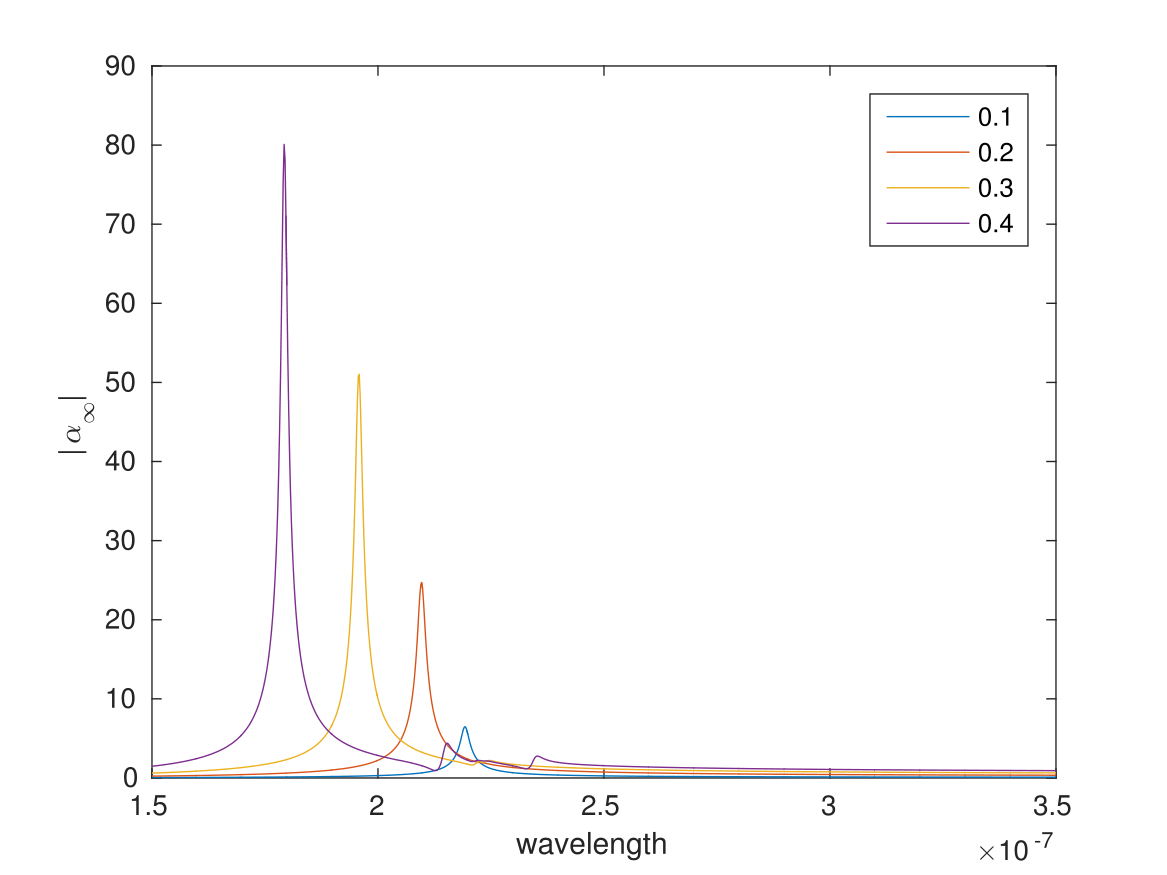}
\caption{ \label{figCirclesShift meta} $|\alpha_{\infty}|$ as a function of the wavelength for disks of different radii, ranging from $0.1$ to $0.4$.}
\end{center}
\end{figure}

\begin{figure}[h!]
\begin{center}
\includegraphics[scale=0.25]{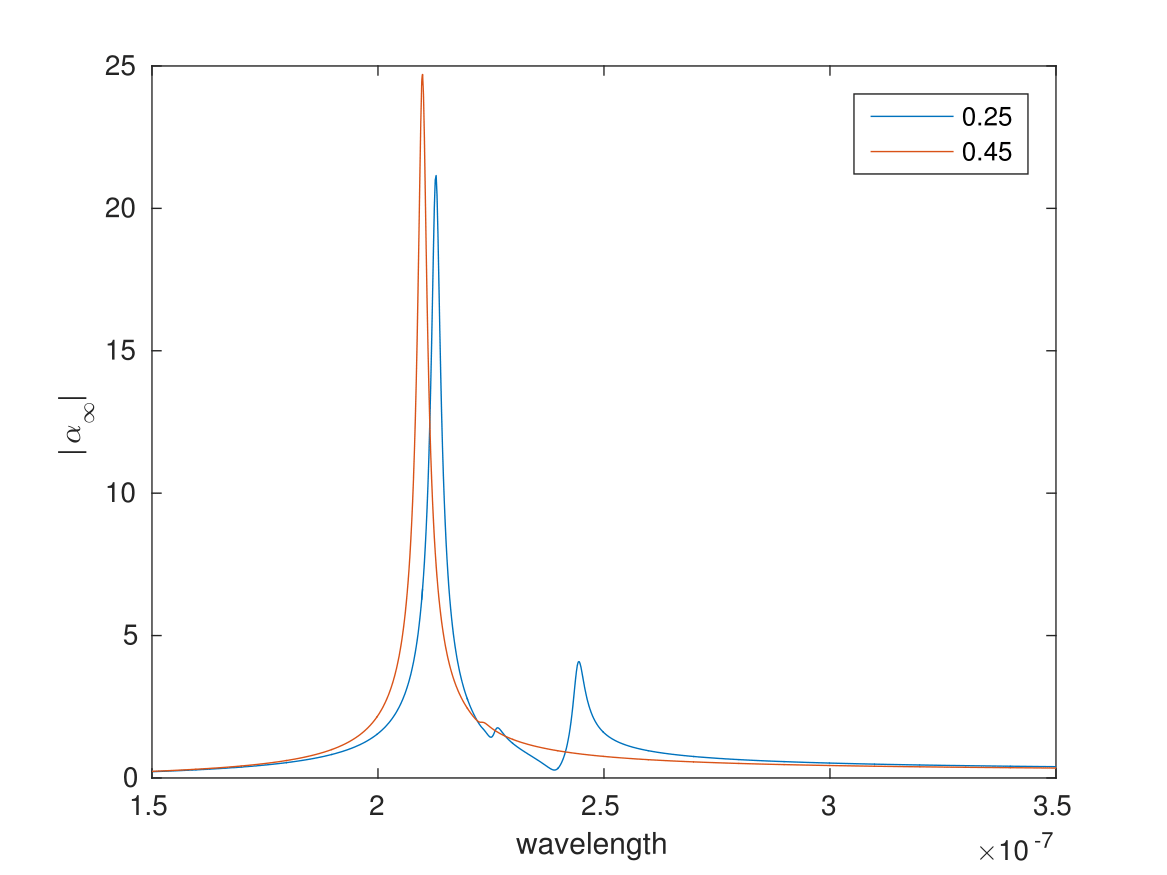}
\caption{ \label{figCirclesVerticalShift meta} $|\alpha_{\infty}|$ as a function of the wavelength for a disk centered respectively at distance $0.25$ and $0.45$ from $x_2=0$.}
\end{center}
\end{figure}

\begin{figure}[h!]
	\begin{center}
		\includegraphics[scale=0.25]{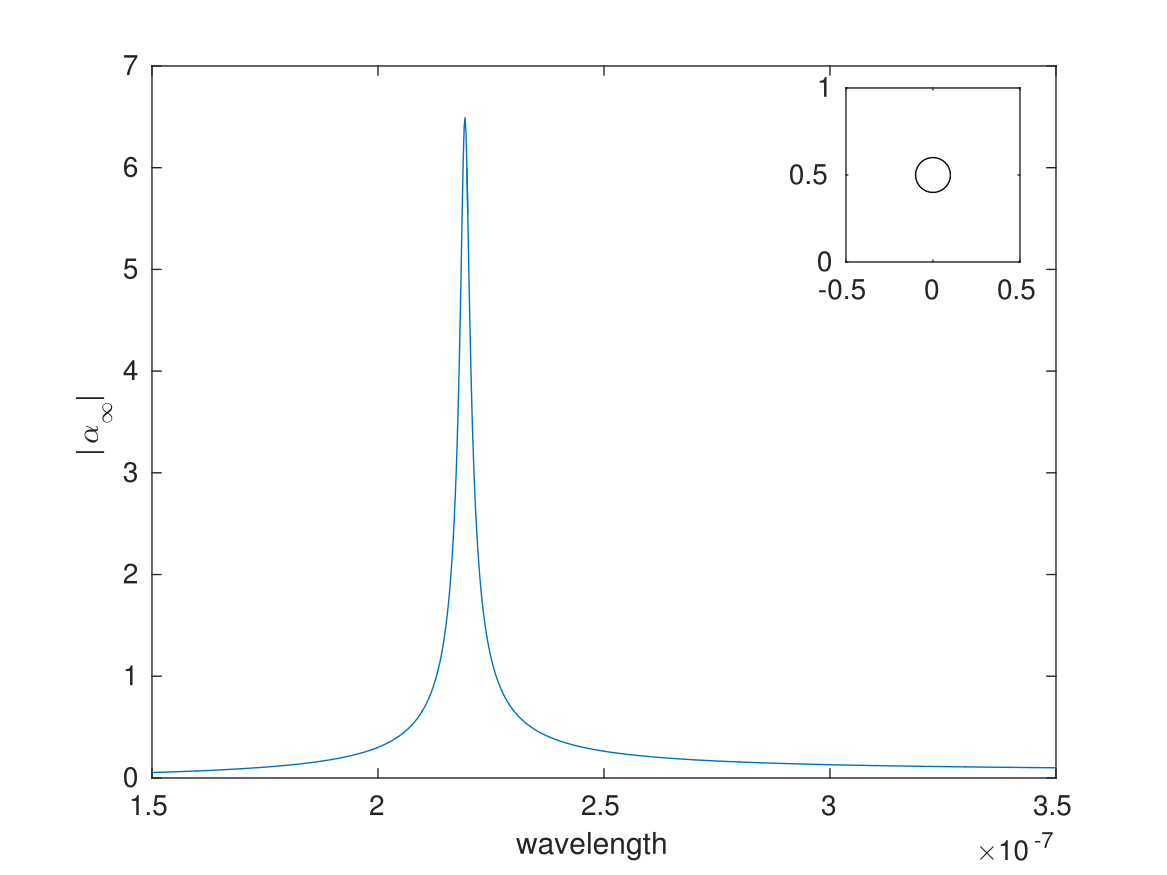}
		\caption{ \label{figCircleSingleFreq meta} Well localized resonance for a disk.}
	\end{center}
\end{figure}

\begin{figure}[h!]
	\begin{center}
		\includegraphics[scale=0.25]{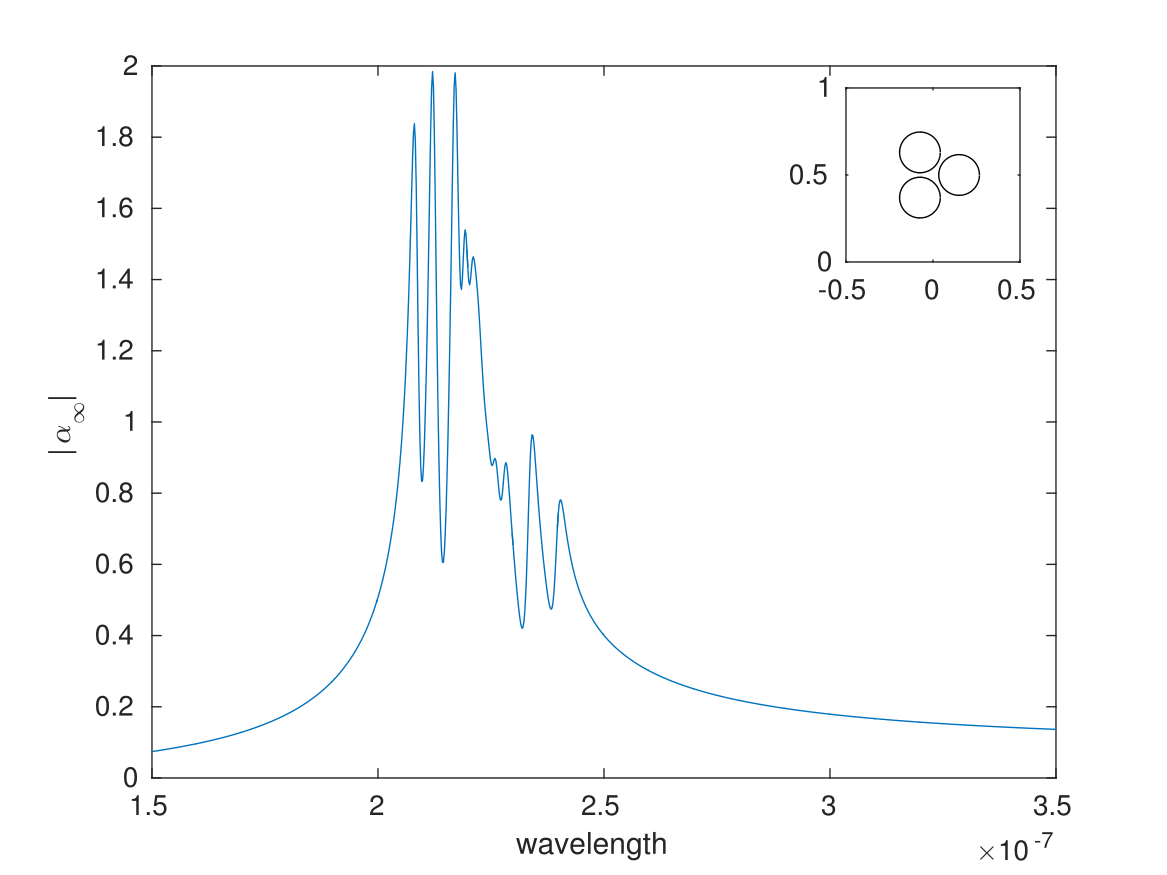}
		\caption{ \label{figThreeCircleSingleFreq meta} Delocalized resonances for three well-separated disks.}
	\end{center}
\end{figure}

\end{document}